\documentclass{amsart}
\usepackage{amsrefs}

\newtheorem{theorem}{Theorem}
\newtheorem{corollary}[theorem]{Corollary}
\newtheorem{lemma}[theorem]{Lemma}
\newtheorem{definition}[theorem]{Definition}

\author{Gord Sinnamon}
\address{Department of Mathematics,
University of Western Ontario,
London, Canada}
\email{sinnamon@uwo.ca}
\thanks{Supported by the Natural Sciences and Engineering Research Council of Canada}

\begin{document}

\title{Norm of the discrete Ces\`aro operator minus identity}

\begin{abstract} The norm of $C-I$ on $\ell^p$, where $C$ is the Ces\`aro operator, is shown to be $1/(p-1)$ when $1<p\le2$. This verifies a recent conjecture of G. J. O. Jameson.  The norm of $C-I$ on $\ell^p$ is also determined when $2< p<\infty$. The two parts together answer a question raised by G. Bennett in 1996. Operator norms in the continuous case, Hardy's averaging operator minus identity, are already known. Norms in the discrete and continuous cases coincide.
\end{abstract}

\maketitle


The Ces\`aro operator, $C$, maps a sequence $(x_n)$ to $(y_n)$, where 
\[
y_n=\frac1n\sum_{k=1}^nx_k.
\]
Hilbert space methods, see \cite{BHS}, show that the operator norm of $C-I$, as a map on $\ell^2$, is 1. The question of determining the exact norm of $C-I$ as a map on $\ell^p$ was posed in 1996 by Grahame Bennett as Problem 10.5 in \cite{GB}. Recently, Jameson \cite{J} answered the question in the case $p=4/3$ by showing that $\|C-I\|_{\ell^{4/3}}=3$.  He conjectured that $\|C-I\|_{\ell^p}=1/(p-1)$ for $1<p\le2$. In Theorems \ref{theorem} and \ref{extn1}, Jameson's method for the case $p=4/3$ is extended and used to verify his conjecture and to answer Bennett's question for the index range $1<p\le2$.

Jameson also gave the upper bound $\|C-I\|_{\ell^4}\le3^{1/4}$ for the operator norm of $C-I$ when $p=4$. In Theorems \ref{upper} and \ref{extn2} we extend the bound to all $p>2$ and show it is best possible. This completes the following answer to Bennett's question: If $1<p\le\infty$, then
\[
\|C-I\|_{\ell^p}=\begin{cases}1/(p-1),&1<p\le 2;\cr m_p^{-1/p},&2<p<\infty;\cr2,&p=\infty.\end{cases}
\]
Here $m_p$ is the minimum value taken by the function $pt^{p-1}+(1-t)^p-t^p$ on the interval $[0,\frac12]$. See Definition \ref{mp} and Lemma \ref{lemma2} below. The minimum value $m_p$ is easy to compute when $p=3$ and $p=4$; the former gives $\|C-I\|_{\ell^3}=(2-\sqrt2)^{-1/3}$ and the latter recovers Jameson's upper bound. The result $\|C-I\|_{\ell^\infty}=2$ appears in \cite{J}*{Proposition 1}.

Hardy's averaging operator takes a function $x$ on $(0,\infty)$ to $Px(s)=\frac1s\int_0^sx$. Many results for $P-I$ were first obtained from work on the Beurling-Ahlfors transform on radial functions. For background and references, see \cite{S17}. The operator $P-I$ has been studied as a map on $L^p=L^p(0,\infty)$, on the positive cone of $L^p$, and on the cone of positive, decreasing functions on $L^p$. Results for weighted $L^p$, see \cite{BS11}, and weak-$L^p$, see \cite{BS19}, are also known. For background and additional references, see \cite{S20}. Special cases of these general results reveal that the values for $\|C-I\|_{\ell^p}$, stated above, coincide with those for $\|P-I\|_{L^p}$.

In corollaries following the main theorems below, we imitate the proofs for $C-I$ to give analogous proofs for $P-I$. These quick, elementary proofs recover the known values of $\|P-I\|_{L^p}$ in the case $1<p<2$ (see \cite[Theorem 4.1]{BJ}, \cite[Theorem 5.3]{BO}, and \cite[(3.13)]{V}) and the case $2<p$ (see \cite[(1.2)]{S17}).

In what follows we consider only real-valued sequences and functions, but, as is well known, extending a linear operator from real to complex values does not change its $\ell^p$ or $L^p$ operator norm.

In several of the arguments below, it is important that the power functions involved be defined on all of $\mathbb R$. Let
\[
\mathbb E=\big\{\tfrac{2i}{2j+1}:i,j\in\mathbb Z\big\}\cap(2,\infty),
\]
a dense subset of $[2,\infty)$. If $p\in\mathbb E$, the function $t\to t^p$ is twice continuously differentiable on $\mathbb R$, it is non-negative, its derivative is strictly increasing and the Mean Value Theorem implies that for all $a,b\in\mathbb R$, 
\begin{equation}\label{mvt}
pa^{p-1}(b-a)\le b^p-a^p\le pb^{p-1}(b-a).
\end{equation}
To investigate $C-I$ in the case $1<p<2$ it is convenient to work instead with the transpose Ces\`aro operator, $C^T$, in the case $2<p<\infty$. The transpose maps a sequence $(x_n)\in\ell^p$ to the sequence $(y_n)$, where 
\[
y_n=\sum_{k=n}^\infty\frac{x_k}k.
\]
Since $\|C-I\|_{\ell^p}=\|C^T-I\|_{\ell^{p'}}$, the conjecture $\|C-I\|_p=p'-1$ for $1<p<2$ may be equivalently stated as $\|C^T-I\|_{\ell^p}=p-1$ for $p\ge2$. (Here and throughout, $1/p+1/p'=1$).

\begin{lemma}\label{lemma1} Let $p\in\mathbb E$. For all $t\in\mathbb R$,
\[
(p-1)p^{p-2}t^p+(t+1)^p-p(t+1)^{p-1}t\ge p^{p-2}(p-1)^{1-p}.
\]
\end{lemma}
\begin{proof} Since $p>2$, $p^{p-2}>1$ so the left-hand side of the inequality goes to infinity as $|t|$ does. Its derivative is $p(p-1)t((pt)^{p-2}-(t+1)^{p-2})$, which vanishes if and only if $t=0$ or $t+1=\pm pt$. This derivative goes from positive to negative as $t$ crosses $0$ so there is a local maximum at $t=0$. The other critical points are $t=1/(p-1)$ and $t=-1/(p+1)$, so the minimum value taken by the left-hand side is the smaller of,
\[
p^{p-2}(p-1)^{1-p}\quad\mbox{and}\quad p^{p-2}(2p-1)(p+1)^{1-p}.
\]
To complete the proof we show that the first of these is smaller. This is equivalent to showing that $h(p)=(p-1)(\log(p+1)-\log(p-1))-\log(2p-1)\le0$. Since $h(2)=0$ it suffices to show that $h'(p)\le0$  for $p>2$. Since $\frac1s$ lies below its secant on $[p-1,p+1]$, we have $\frac1s\le\frac{2p-s}{p^2-1}$. Thus $\log(p+1)-\log(p-1)\le\int_{p-1}^{p+1}\frac{2p-s}{p^2-1}\,ds=\frac{2p}{p^2-1}$. Also, $2p-1\le p^2-1$ so
\[
h'(p)\le\frac{2p}{p^2-1}+(p-1)\Big(\frac1{p+1}-\frac1{p-1}\Big)-\frac2{p^2-1}=0.\qedhere
\]
%
%
%
%
%
\end{proof}

Now we are ready to show that $\|C^T-I\|_{\ell^p}\le p-1$ for $p\in\mathbb E$. The proof follows the method of \cite{J}*{Theorem 2}.

\begin{theorem}\label{theorem} Let $p\in\mathbb E$, let $(x_n)\in\ell^p$ be a real sequence, and set $y_n=\sum_{k=n}^\infty \frac{x_k}k$. Then
\[
\sum_{n=1}^\infty (y_n-x_n)^p\le(p-1)^p\sum_{n=1}^\infty x_n^p.
\]
\end{theorem}
\begin{proof} By H\"older's inequality, the sum defining $y_N$ converges absolutely and 
\[
|y_N|\le\sum_{n=N}^\infty\frac{|x_n|}n\le\Big(\sum_{n=N}^\infty x_n^p\Big)^{1/p}\Big(\sum_{n=N}^\infty n^{-p'}\Big)^{1/p'}\sim\Big(\sum_{n=N}^\infty x_n^p\Big)^{1/p}N^{-1/p}.
\]
Therefore, as $N\to\infty$,
\[
\sum_{n=1}^N \big(ny_{n+1}^p-(n-1)y_n^p\big) =Ny_{N+1}^p\to0.
\]
Note that for all $n$, $x_n=n(y_n-y_{n+1})$. For each $n$, (\ref{mvt}) implies
\[
ny_{n+1}^p-(n-1)y_n^p=y_n^p-n(y_n^p-y_{n+1}^p)
\ge y_n^p-py_n^{p-1}n(y_n-y_{n+1})=y_n^p-py_n^{p-1}x_n.
\]

Now let $z_n=y_n-x_n$ and set $c=p^{p-2}(p-1)^{1-p}$. If $z_n\ne0$ we may let $t=x_n/z_n$ and apply Lemma \ref{lemma1} in the form $(t+1)^p-p(t+1)^{p-1}t\ge c(1-(p-1)^pt^p)$ to get
\[
y_n^p-py_n^{p-1}x_n=z_n^p\big((1+t)^p-p(1+t)^{p-1}t\big)
\ge cz_n^p\big(1-(p-1)^pt^p\big)
=c(z_n^p-(p-1)^px_n^p).
\]
 If $z_n=0$,  then $y_n=x_n$ and the same inequality holds, as
\[
y_n^p-py_n^{p-1}x_n=-(p-1)x_n^p\ge-p^{p-2}(p-1)x_n^p=c(z_n^p-(p-1)^px_n^p).
\]
Summing over $n$, we have
\[
c\sum_{n=1}^N \big(z_n^p-(p-1)^px_n^p\big)\le \sum_{n=1}^N \big(ny_{n+1}^p-(n-1)y_n^p\big) \to0
\]
as $N\to\infty$ and we conclude that 
\[
\sum_{n=1}^\infty z_n^p\le(p-1)^p\sum_{n=1}^\infty x_n^p.\qedhere
\]
\end{proof}

The proof simplifies in the continuous case, where instead of the transpose Ces\`aro operator we work with the dual averaging operator $P^T$. 
\begin{corollary} Let $p\in\mathbb E$, let $x\in L^p$, and set $y(s)=P^Tx(s)=\int_s^\infty x(\theta)\,\frac{d\theta}\theta$. Then $\int_0^\infty (y-x)^p\le (p-1)^p\int_0^\infty x^p$.
\end{corollary}
\begin{proof} It suffices to establish the result for functions $x$ in a dense subset of $L^p$ so suppose $x$ is continuous and compactly supported in $(0,\infty)$. Then $y=P^Tx$ satisfies $\lim_{s\to0^+}sy(s)^p=0=\lim_{s\to\infty}sy(s)^p$ and $\frac{d}{ds}(sy(s)^p)=y(s)^p-py(s)^{p-1}x(s)$ so $\int_0^\infty (y^p-py^{p-1}x)=0$. Let $z=y-x$ and $c=p^{p-2}(p-1)^{1-p}$. If $z\ne0$ and $t=x/z$, then Lemma \ref{lemma1} implies
\[
c(z^p-(p-1)^px^p)=cz^p(1-(p-1)^pt^p)\le z^p((1+t)^p-p(1+t)^{p-1}t)=y^p-py^{p-1}x,
\]
which also holds when $z=0$. Integrate to get $\int_0^\infty z^p\le(p-1)^p\int_0^\infty x^p$.
\end{proof}

Next we extend Theorem \ref{theorem} from $p\in\mathbb E$ to all $p>2$ and point out a known lower bound for $\|C^T-I\|_{\ell^p}$.

\begin{theorem}\label{extn1} If $2\le p<\infty$ then $\|C^T-I\|_{\ell^p}=p-1$ and if $1<p\le2$ then $\|C-I\|_{\ell^p}=p'-1=1/(p-1)$.
\end{theorem}
\begin{proof} As mentioned above, the case $p=2$ is known to hold. For $p>2$, Theorem \ref{theorem} shows that for $p\in \mathbb E$ with $2<p<\infty$, $\|C^T-I\|_{\ell^p}\le p-1$. The Riesz-Thorin Theorem, see \cite{BS}*{Corollary IV.2.3}, implies that if $2<p_0<p<p_1<\infty$, with $p_0,p_1\in\mathbb E$, then for some $\theta\in(0,1)$, depending on $p_0$, $p$, and $p_1$,
$$
\|C^T-I\|_{\ell^p}\le\|C^T-I\|_{\ell^{p_0}}^{1-\theta}\|C-I\|_{\ell^{p_1}}^\theta=(p_0-1)^{1-\theta}(p_1-1)^\theta\le p_1-1.
$$
Letting $p_1\to p$ through $\mathbb E$, we get $\|C^T-I\|_{\ell^p}\le p-1$. 

The dual discrete Hardy inequality, \cite{HLP}*{Theorem 331}, shows that for $p>1$, 
$\|C^T\|_{\ell^p}=p$. Therefore, $\|C^T-I\|_{\ell^p}\ge\|C^T\|_{\ell^p}-\|I\|_{\ell^p}=p-1$ and we conclude that $\|C^T-I\|_{\ell^p}= p-1$ for all $p>2$. The second statement of the theorem follows from the first by duality.
\end{proof}

The continuous case follows in just the same way because $\|P^T\|_{L^p}=p$. See \cite{HLP}*{Theorem 328}. The proof is omitted.
\begin{corollary} If $2\le p<\infty$ then $\|P^T-I\|_{L^p}=p-1$ and if $1<p\le2$ then $\|P-I\|_{L^p}=p'-1=1/(p-1)$.
\end{corollary}

Next we consider the case $p>2$. To begin we introduce $m_p$, essential for our formula for the operator norm of $C-I$.

\begin{definition}\label{mp} Let $p\ge2$ and set $f_p(t)=pt^{p-1}+(1-t)^p-t^p$. Define $m_p$ to be the minimum value of $f_p(t)$ for $0\le t\le \frac12$.
\end{definition}

\begin{lemma}\label{lemma2} If $p>2$, then $f_p$ has a unique critical point $t_p$ in $(0,\frac12)$, $m_p=f_p(t_p)$ and $m_p$ is a continuous function of $p$. If, in addition, $p\in\mathbb E$, then $t_p$ is the unique critical point of $f_p$ on all of $\mathbb R$ and $f_p(t)\ge m_p$ for all $t\in\mathbb R$.
\end{lemma}
\begin{proof} On $(0,\frac12)$, $f_p'(t)=p((p-1)t^{p-2}-(1-t)^{p-1}-t^{p-1})$. It extends to be continuous on $[0,\frac12]$ with $f_p'(0)=-p<0$, and $f_p'(\frac12)=p(p-2)2^{2-p}>0$. On $(0,\frac12)$, $0<t<1-t$ so $f_p''(t)=p(p-1)((p-2)t^{p-3}+(1-t)^{p-2}-t^{p-2})>0$.

Therefore $f_p'$ is strictly increasing on $[0,\frac12]$, $f_p$ has a unique critical point $t_p$ in $(0,\frac12)$ and $f_p(t_p)$ is the minimum value of $f_p$, namely $m_p$. For any $p_0>2$, the function $(p,t)\mapsto f_p(t)$ is uniformly continuous on $[2,p_0]\times[0,\frac12]$. It follows that $p\mapsto m_p$ is continuous on $[2,\infty)$.


If $p\in\mathbb E$, then $f_p$ is defined on $\mathbb R$, $(1-t)^p=(t-1)^p$, and 
\[
f_p'(t)=p((p-1)t^{p-2}+(t-1)^{p-1}-t^{p-1})=p(p-1)\int_{t-1}^t(t^{p-2}-s^{p-2})\,ds.
\]
If $t\le0$ then $t-1<s<t$ implies $|t|<|s|$ so $t^{p-2}<s^{p-2}$ and we have $f_p'(t)<0$. Thus, $f_p$ is strictly decreasing on $(-\infty,0]$. If $t\ge\frac12$ then $t-1<s<t$ implies $|s|<t$ so $s^{p-2}<t^{p-2}$ and we have $f_p'(t)>0$. Thus $f_p$ is strictly increasing on $[\frac12,\infty)$. It follows that $t_p$ is the unique critical point of $f_p$ on $\mathbb R$ and $f_p(t)\ge m_p$ for all $t\in \mathbb R$.
\end{proof}


In \cite{J}*{Theorem 1}, the upper bound $\|C-I\|_{\ell^4}\le3^{1/4}$ was proved. We employ a similar method to extend it to an upper bound for all $p>2$. 

\begin{theorem}\label{upper} Let $p\in\mathbb E$, let $(x_n)\in\ell^p$ be a real sequence, and set $y_n=\frac1n\sum_{k=1}^n x_k$. Then
\[
\sum_{k=1}^\infty (y_k-x_k)^p\le \frac1{m_p}\sum_{k=1}^\infty x_k^p.
\]
\end{theorem}
\begin{proof} Fix a $y_0$ arbitrarily and observe that $(n-1)(y_{n-1}-y_n)=y_n-x_n$ for $n=1,2,\dots$. By (\ref{mvt}), 
\[
(n-1)(y_{n-1}^p-y_n^p)\ge (n-1)py_n^{p-1}(y_{n-1}-y_n)=py_n^{p-1}(y_n-x_n).
\]
This becomes
\[
ny_n^p-(n-1)y_{n-1}^p\le py_n^{p-1}x_n-(p-1)y_n^p=y_n^{p-1}(px_n-(p-1)y_n).
\]
Take $z_n=y_n-x_n$ and $t=y_n/z_n$ to get $px_n-(p-1)y_n=z_n(t-p)$. By Lemma \ref{lemma2},
\[
ny_n^p-(n-1)y_{n-1}^p\le z_n^p(t^p-pt^{p-1})\le z_n^p((t-1)^p-m_p)=x_n^p-m_pz_n^p.
\]
Summing from $1$ to $N$ gives
\[
0\le Ny_N^p=\sum_{n=1}^N(ny_n^p-(n-1)y_{n-1}^p)\le\sum_{n=1}^Nx_n^p -m_p\sum_{n=1}^Nz_n^p.
\]
Letting $N\to\infty$ we have 
\[
\sum_{n=1}^\infty z_n^p\le\frac1{m_p}\sum_{n=1}^\infty x_n^p.\qedhere
\]
\end{proof}

\begin{corollary} Let $p\in\mathbb E$, let $x\in L^p$ and set $y(s)=Px(s)=\frac1s\int_0^sx$. Then $\int_0^\infty(y-x)^p\le\frac1{m_p}\int_0^\infty x^p$.
\end{corollary}
\begin{proof} It suffices to establish the result for functions $x$ in a dense subset of $L^p$ so suppose $x$ is continuous and compactly supported in $(0,\infty)$. Then $y=Px$ satisfies $\lim_{s\to0^+}sy(s)^p=0=\lim_{s\to\infty}sy(s)^p$ and $\frac{d}{ds}(sy(x)^p)=(1-p)y(s)^p+py(s)^{p-1}x(s)$ so $\int_0^\infty ((1-p)y^p+py^{p-1}x)=0$. Let $z=y-x$. If $z\ne0$ and $t=y/z$, then Lemma \ref{lemma2} implies
\[
x^p-m_pz^p=z^p((t-1)^p-m_p)\ge z^p(t^p-pt^{p-1})=(1-p)y^p+py^{p-1}x,
\]
which also holds when $z=0$. Integrate to get $\int_0^\infty z^p\le\frac1{m_p}\int_0^\infty x^p$.
\end{proof}

We again pass from $p\in\mathbb E$ to all $p>2$ using the Riesz-Thorin Theorem. However, this time the lower bound requires some work.

\begin{theorem}\label{extn2} If $p\ge2$ then $\|C-I\|_{\ell^p}= m_p^{-1/p}$.
\end{theorem}
\begin{proof} It is easy to verify that $m_2=1$, so the case $p=2$ agrees with the known result. Now suppose $p>2$. Theorem \ref{upper} shows that $\|C-I\|_{\ell^p}\le m_p^{-1/p}$ for all $p\in\mathbb E$. The Riesz-Thorin theorem implies that if $2<p_0<p<p_1<\infty$, with $p_0,p_1\in\mathbb E$, then for some $\theta\in(0,1)$, depending on $p_0$, $p$, and $p_1$,
$$
\|C-I\|_{\ell^p}\le\|C-I\|_{\ell^{p_0}}^{1-\theta}\|C-I\|_{\ell^{p_1}}^\theta
\le\max(m_{p_0}^{-1/p_0},m_{p_1}^{-1/p_1}).
$$
Letting $p_0$ and $p_1$ approach $p$ through $\mathbb E$, the continuity of $p\mapsto m_p$ implies that $\|C-I\|_{\ell^p}\le m_p^{-1/p}$.

To prove the reverse inequality we set $r=1/t_p$, where $t_p$ is the critical point from Lemma \ref{lemma2}, and fix an integer $m>1$. Note that $r>2$. Define $x_n=-m^{-r}$ for $n\le m$ and $x_n=(n-1)^{1-r}-n^{1-r}$ for $n>m$. Then, with $y_n=\frac1n\sum_{k=1}^nx_k$ and $z_n=y_n-x_n$,
\[
y_n=\begin{cases}-m^{-r},&n\le m;\cr -n^{-r},&n\ge m\end{cases}\quad\mbox{and}\quad
z_n=\begin{cases}0,&n\le m;\cr n^{1-r}-(n-1)^{1-r}-n^{-r},&n>m\end{cases}
\]
If $n\ge m+1$, then
\[
0<x_n=(r-1)\int_{n-1}^nt^{-r}\,dt\le (r-1)(n-1)^{-r}
\]
so, employing a standard Riemann sum estimate, we get
\[
\sum_{n=1}^\infty|x_n|^p\le m^{1-pr}+(r-1)^p\sum_{n=m+1}^\infty(n-1)^{-pr}
\le m^{1-pr}+(r-1)^p\frac{(m-1)^{1-pr}}{pr-1}.
\]
Also, if $n\ge m+1$, then $\frac{n-1}n\ge\frac m{m+1}$ and
\[
-z_n=r(n-1)\int_{n-1}^nt^{-r-1}\,dt\ge r(n-1)n^{-r-1}\ge r\frac m{m+1}n^{-r}>0
\]
so, using another standard Riemann sum estimate, we get
\[
\sum_{n=1}^\infty|z_n|^p\ge r^p\Big(\frac m{m+1}\Big)^p\sum_{n=m+1}^\infty n^{-pr}
\ge r^p\Big(\frac m{m+1}\Big)^p\frac{(m+1)^{1-pr}}{pr-1}.
\]
We conclude that 
\[
\frac{\sum_{n=1}^\infty|z_n|^p}{\sum_{n=1}^\infty|x_n|^p}
\ge\frac{r^p\big(\frac m{m+1}\big)^p}{(pr-1)\big(\frac m{m+1}\big)^{1-pr}+(r-1)^p\big(\frac {m-1}{m+1}\big)^{1-pr}}\to\frac{r^p}{(pr-1)+(r-1)^p}
\]
as $m\to\infty$. Since $r=1/t_p$, Lemma \ref{lemma2} shows that the last expression is $1/m_p$. This implies that the operator norm of $C-I$ on $\ell^p$ cannot be less than $m_p^{-1/p}$, which completes the proof.
\end{proof}

\begin{corollary} If $p\ge2$ then $\|P-I\|_{L^p}= m_p^{-1/p}$.
\end{corollary}
\begin{proof} The upper bound is extended from $p\in\mathbb E$ to all $p>2$ just as in the last theorem. However, proving the reverse inequality is much simpler. With $t_p$ as in Lemma \ref{lemma2} and $r=1/t_p$, let $x(s)=-1$ on $(0,1)$ and $x(s)=(r-1)s^{-r}$ on $(1,\infty)$. With $y=Px$ and $z=y-x$ we compute $y(s)=-1$ on $(0,1)$ and $y(s)=-s^{-r}$ on $(1,\infty)$; and $z(s)=0$ on $(0,1)$ and $z(s)=-rs^{-r}$ on $(1,\infty)$. Then 
\[
\int_0^\infty|x|^p=1+\frac{(r-1)^p}{pr-1}\quad\mbox{and}\quad 
\int_0^\infty|z|^p=\frac{r^p}{pr-1}.
\]
It follows that $\int_0^\infty|z|^p=\frac1{m_p}\int_0^\infty|x|^p$, which gives the lower bound.
\end{proof}

The expressions for $\|P-I\|_{L^p}$ given above and in \cite{S17} must coincide, but a direct connection is still worth making: In \cite{S17}, we find $\|P-I\|_{L^p}^p=\sup_{\alpha\le p'}\frac{|\alpha-1|^p}{p(1-\alpha)-1+|\alpha|^p}$. With $t=1/(1-\alpha)$ this expression readily reduces to $1/m_p$ for $p\in\mathbb E$, by applying Lemma \ref{lemma2}. Equality for all $p$ follows by continuity.

Acknowledgements. My thanks to Santiago Boza for providing references to results in the continuous case and to Graham Jameson for an improvement to the proof of Lemma \ref{lemma1}.

\bibliographystyle{amsplain}

\begin{bibdiv}
\begin{biblist}
\bib{BJ}{article}{
   author={Ba\~{n}uelos, Rodrigo},
   author={Janakiraman, Prabhu},
   title={On the weak-type constant of the Beurling-Ahlfors transform},
   journal={Michigan Math. J.},
   volume={58},
   date={2009},
   number={2},
   pages={459--477},
   issn={0026-2285},
   review={\MR{2595549}},
   doi={10.1307/mmj/1250169072},
}
\bib{BO}{article}{
   author={Ba\~{n}uelos, Rodrigo},
   author={Os\c{e}kowski, Adam},
   title={Sharp inequalities for the Beurling-Ahlfors transform on radial
   functions},
   journal={Duke Math. J.},
   volume={162},
   date={2013},
   number={2},
   pages={417--434},
   issn={0012-7094},
   review={\MR{3018958}},
   doi={10.1215/00127094-1962649},
}
\bib{BS}{book}{
   author={Bennett, Colin},
   author={Sharpley, Robert},
   title={Interpolation of operators},
   series={Pure and Applied Mathematics},
   volume={129},
   publisher={Academic Press, Inc., Boston, MA},
   date={1988},
   pages={xiv+469},
   isbn={0-12-088730-4},
   review={\MR{928802}},
}
\bib{GB}{article}{
   author={Bennett, Grahame},
   title={Factorizing the classical inequalities},
   journal={Mem. Amer. Math. Soc.},
   volume={120},
   date={1996},
   number={576},
   pages={viii+130},
   issn={0065-9266},
   review={\MR{1317938}},
   doi={10.1090/memo/0576},
}
\bib{BS11}{article}{
   author={Boza, Santiago},
   author={Soria, Javier},
   title={Solution to a conjecture on the norm of the Hardy operator minus
   the identity},
   journal={J. Funct. Anal.},
   volume={260},
   date={2011},
   number={4},
   pages={1020--1028},
   issn={0022-1236},
   review={\MR{2747011}},
   doi={10.1016/j.jfa.2010.11.013},
}
\bib{BS19}{article}{
   author={Boza, Santiago},
   author={Soria, Javier},
   title={Averaging operators on decreasing or positive functions:
   equivalence and optimal bounds},
   journal={J. Approx. Theory},
   volume={237},
   date={2019},
   pages={135--152},
   issn={0021-9045},
   review={\MR{3868629}},
   doi={10.1016/j.jat.2018.09.001},
}
	
\bib{BHS}{article}{
   author={Brown, Arlen},
   author={Halmos, P. R.},
   author={Shields, A. L.},
   title={Ces\`aro operators},
   journal={Acta Sci. Math. (Szeged)},
   volume={26},
   date={1965},
   pages={125--137},
   issn={0001-6969},
   review={\MR{187085}},
}
\bib{HLP}{book}{
   author={Hardy, G. H.},
   author={Littlewood, J. E.},
   author={P\'{o}lya, G.},
   title={Inequalities},
   series={Cambridge Mathematical Library},
   note={Reprint of the 1952 edition},
   publisher={Cambridge University Press, Cambridge},
   date={1988},
   pages={xii+324},
   isbn={0-521-35880-9},
   review={\MR{944909}},
}
\bib{J}{article}{
   author={Jameson, Graham J. O.},
   title={The $\ell_p$-norm of $C-I$, where $C$ is the Ces\`aro operator},
   journal={Math. Inequal. Appl.},
   volume={24},
   date={2021},
   number={2},
   pages={551--557},
   issn={1331-4343},
   review={\MR{4254470}},
   doi={10.7153/mia-2021-24-38},
}
\bib{S17}{article}{
   author={Strzelecki, Micha\l },
   title={The $L^p$-norms of the Beurling-Ahlfors transform on radial
   functions},
   journal={Ann. Acad. Sci. Fenn. Math.},
   volume={42},
   date={2017},
   number={1},
   pages={73--93},
   issn={1239-629X},
   review={\MR{3558516}},
   doi={10.5186/aasfm.2017.4204},
}
\bib{S20}{article}{
   author={Strzelecki, Micha\l },
   title={Hardy's operator minus identity and power weights},
   journal={J. Funct. Anal.},
   volume={279},
   date={2020},
   number={2},
   pages={108532, 34},
   issn={0022-1236},
   review={\MR{4088501}},
   doi={10.1016/j.jfa.2020.108532},
}
\bib{V}{article}{
   author={Volberg, A},
   title={Ahlfors–Beurling operator on radial functions},
   journal={arXiv:1203.2291},
   date={2012},
   pages={16},
}

\end{biblist}
\end{bibdiv}

\end{document}